\documentclass[12pt]{article}

\usepackage{amsfonts}
\usepackage{mathrsfs}
\usepackage{amsfonts}
\usepackage{amsfonts}
\usepackage{amsfonts}
\usepackage{amsfonts}
\usepackage{amsfonts}
\usepackage{amsfonts}
\usepackage{amsfonts}
\usepackage{amsfonts}
\usepackage{amsfonts}
\usepackage{amsfonts}
\usepackage{amsfonts}
\usepackage{amsfonts}
\usepackage{amsfonts}
\usepackage{amsfonts}
\usepackage{amscd}
\usepackage{bbm}
\usepackage [dvips]{graphics}
\usepackage[centertags]{amsmath}
\usepackage{amsfonts}
\usepackage{amssymb}
 
\usepackage{amsthm}
\usepackage{newlfont}
\usepackage{stmaryrd}
\usepackage[]{titletoc}  % from here
\usepackage{mathrsfs}
\usepackage{stmaryrd}
\usepackage [dvips]{graphics}
\usepackage[centertags]{amsmath}
\usepackage{amsfonts}
\usepackage{amssymb}
\usepackage{amsthm}
\usepackage{newlfont}
\usepackage{hyperref}

\newlength{\defbaselineskip}
\newcommand{\setlinespacing}[1]%
           {\setlength{\baselineskip}{#1 \defbaselineskip}}

\theoremstyle{plain}
\newtheorem{thm}{Theorem}[section]

\newtheorem{lem}[thm]{Lemma}

\newtheorem{prop}[thm]{Proposition}

\newtheorem{rem}[thm]{Remark}
\newtheorem{Def}[thm]{Definition}

\newcommand{\cc}{\mathbb{C}^m}
\newcommand{\ccc}{\mathbb{C}^n}

\newcommand{\et}{E(\mathbf{T})}

\newcommand{\w}{\widetilde}

\newcommand{\p}{\partial}
\newcommand{\bnn}{\mathcal {B}_n^m(\Omega)}

\makeatletter\@addtoreset{equation}{section} \makeatother
\begin{document}

\textwidth=13.5cm
  \textheight=23cm
  \hoffset=-1cm

  \baselineskip=17pt
\title {Specht's invariant and  localization of operator tuples}
\author{Li Chen}
\date{}
\maketitle {\textbf{}
\def\zz{\textbf{z}}
\def\ww{\textbf{w}}
\def\zzz{\textbf{z}_0}
\def\bg{\mbox{\boldmath$\gamma$}}
\def\ba{\mbox{\boldmath$a$}}
\def\bb{\mbox{\boldmath$b$}}
\def\bet{\mbox{\boldmath$\beta$}}
\def\M{\mathcal{M}}
%action of Ni; block-entry
%
%staring from 0th; right left action; A otimes In
%
%how to say NiNj* restriction
%
%widetilde N*
%
%Z=? submanifold
%
%intertwine Ni or Ti?
%
%definition of $H_Z^2$
%
%definition of Xij-curvature operator
%
%joint kernel
%
%Ez EkZ

\textbf{Abstract:} We  introduce   a  set  of operator-valued   invariants  for   localization of  \linebreak operator  tuples in the Cowen-Douglas class, with which   a Specht-type  criterion for  unitary equivalence  is obtained.

\emph{Key words}:  Specht's invariant; localization; Cowen-Douglas operator
\section{Introduction}

Finding suitable  invariants to classify  non-normal Hilbert space operators  up to unitary equivalence
% Two operator tuples $\mathbf{T}=(T_1,\cdots, T_m)$ and $\mathbf{\w T}=(\w T_1,\cdots, \w T_m)$ and said to be unitarily equivalent if there exists a unitary operator $U$ intertwining $\mathbf{T}$ and $\mathbf{\w T}$, that is, $UT_i=\w T_iU$ for every $1\leq i\leq m$.
%. such as  the normal elements in the Calkin algebra(i.e,essentially normal operators),  which was successfully classified by the highly-regarded BDF theory via methods  from algebraic topology in terms of essential  spectrum and index[].
  is    in general  a widely open and appealing topic. A basic   existence  theorem of scaler-valued  invariants for finite matrices  was given by Specht      in terms of matrix  traces:
  \begin{thm}\label{spt}(Specht \cite{Spe}) Two complex $n\times n$ matrices $A$ and $B$ are unitarily equivalent if and only if
$tr(w(A, A^*)) = tr(w(B, B^*))$  for all  words $w$ in two variables.\end{thm}

  % motivating numerous later works.
On infinite dimensional Hilbert spaces where scaler invariants are not always obtainable, a natural idea  is to consider  operator-valued alternatives.  By  operator-valued invariants  we mean that one      associates    to  every operator  $T$(in a given operator class)      a  set $\mathcal{I}(T)$ consisting  of ``testing operators" acting on certain ``testing spaces",    such that  unitary equivalence of $T$    can be reduced to   unitary equivalences of  operators in  $\mathcal{I}(T)$.
A   well-known   example \cite{NF} is the   complete non-unitary  contractions, whose unitary equivalence can be determined by the so-called  characteristic function  as a family of  testing operators     parameterized by points in the unit   disc acting on the  defect spaces of the contractions.

 \begin{rem}\label{test} Operator-valued invariants  always exist(one can trivially set $\mathcal{I}(T)=\{T\}$ as a singleton with  entire space as the testing space),  and   conceptual  non-triviality lies in     the existence of    small and canonical testing spaces
%, involving the interplay between operator theory and other mathematical branches such as complex analysis,commutative algebra, algebraic topology, sheaf theory, K-theory[][][][],etc.  As it is  fascinating to  establish abstract relations,    the more appealing thing is   employing   methods from a different field  to  solve  concrete  problems in operator theory.
  (in particular,     existence of  scaler  invariants follows from existence of one dimensional  testing spaces). Minimizing  the cardinality of $\mathcal{I}(T)$, after    testing spaces  identified and fixed,  is a      technical problem (see, for instance, a series of refinements  \cite{Pear,Pap,DJJ}  on  Theorem \ref{spt})  and will not be the theme of  this paper.\end{rem}

 % Specht's Theorem for a class of non-normal operators on infinite dimensional Hilbert spaces, which for the first time   provides  a complete set of scalar-valued unitary invariants, as the starting point from which
%
%it is natural to ask for an  analogue of  for non-normal operators on infinite dimensional Hilbert spaces.p
In this paper we work on multi-variate case with  a tuple of commuting operators \linebreak lying in the following  important and extensively studied  class    introduced  by M. Cowen and R.Douglas \cite{CD,CD2}.

   \begin{Def}\label{bn}Given  positive integers $m,n$, and a bounded domain $\Omega$ in $\cc$, a commuting $m$-tuple of operators $\textbf{T}=(T_1,\cdots,T_m)$ acting on a Hilbert space $\mathcal {H}$ belongs to the   class    $\mathcal {B}_n^m(\Omega) $ if the followings  hold:

 $(i)$ The   space $\{((T_1-z_1)h,\cdots, (T_m-z_m)h), h\in \mathcal {H}\}$  is a closed subspace in  $\mathcal {H}\oplus\cdots\oplus \mathcal {H}$($m$ copies of $\mathcal {H}$) for every $\zz=(z_1,\cdots,z_m)\in\Omega$;

 $(ii)$ $\dim\cap_{i=1}^m \ker(T_i-{z_i})=n$ for every   $\zz\in\Omega$ and

  $(iii)$ $ \vee_{\zz\in\Omega}\cap_{i=1}^m \ker(T_i-{z_i})=\mathcal {H}$(here $\vee$ denotes the closed linear span).
\end{Def}

  %
% the ``spanning property" (iii) suggests that the global property of $\mathbf{T}$ might be detected by    ``localizations" parameterized by points in $\Omega$.
 Operator tuples in $\mathcal {B}_n^m(\Omega) $ can be modeled by   adjoints of coordinate multiplications on $\ccc$-valued    holomorphic function spaces  in $m$ complex variables \cite{CS} such as Hardy or Bergman   spaces, making it a rich class of considerable interest.
%A unitary operator $U$ intertwining  $\mathbf{T}$ and $\mathbf{\w T}$ in $\mathcal {B}_n^m(\Omega)$ naturally gives, for every $z\in\Omega$, the  unitary equivalence between    $\mathbf{T}|_{H_\zz^{k}}$ and $\mathbf{\w T}|_{\w H_\zz^{k}}$  via $U|_{H_\zz^{k}}$, while
%  to conversely construct $U$ from given  (pointwise)local equivalences  is highly non-trivial which, as  the main concern  of Cowen and Douglas's important seminal paper[],
 %  focusing on an operator-theoretic characterization, which bears advantages from both Theorem \ref{rig} and \ref{cdmain} which involves only first order localizations and does not assume a gluing fashion.
%In light of Definition \ref{bn}, the  natural candidates for testing spaces   of   $\mathcal {B}_n^m(\Omega) $ are $n$-dimensional joint eigen-spaces  $H_\zz^{ 1}$, while         in most   literatures   they are  viewed    as a geometric entity   arranged into the holomorphic vector bundle $E^1({\mathbf{T}})$,  from  which  one usually gets geometric invariants rather than   operator-theoretic invariants.
%\textbf{{Rigidity Theorem}}
%\emph{Operator tuples $\mathbf{T} $ and $\mathbf{\w T}$ in $\mathcal {B}_n^m(\Omega) $  are unitarily equivalent if and only if there exists a holomorphic isometric bundle map between $E({\mathbf{\w T}})$
%and $E({\mathbf{T}})$.}
%\bigskip
%As  numerous   works  have been done on the geometric theory   of  $\mathcal {B}_n^m(\Omega) $ via $\et$(see [][] for some recent works  and [] for a nice survey containing more references),
We record the following theorem  which is the main  result  of   Cowen and Douglas' seminal work, asserting that unitary equivalence of operator tuples  in $\mathcal {B}_n^m(\Omega) $ can be tested on    a family of   finite dimensional   spaces.

 %In most works related to the Cowen-Douglas theory, localizations are viewed as a geometric entity   arranged into the holomorphic vector bundle $E^k({\mathbf{T}})$, from  which  one generally get geometric invariants rather than   operator-theoretic invariants.

    \begin{thm}\label{cdmain}(\cite{CD1,CD})Operator tuples $\mathbf{T} $ and $\mathbf{\w T}$ in $\mathcal {B}_n^m(\Omega) $  are unitarily equivalent if and only if $\mathbf{T}|_{H_\zz^{n+1}}$ is unitarily equivalent to $\mathbf{\w T}|_{\w H_\zz^{n+1}}$ for every $\zz$ in $\Omega$.\end{thm}

 Here the space

 $$H_\zz^k:= \cap_{|I|=k}\ker (\mathbf{T}-\zz)^I$$ is called  the  \emph{$k$-th order localization} of $\mathbf{T}$ at  $\zz=(z_1,\cdots,z_m)$, where $$(\mathbf{T}-\zz)^I:=(T_1-z_1)^{i_1}(T_2-z_2)^{i_2}\cdots(T_m-z_m)^{i_m},$$ $I=(i_1,\cdots,i_m)$ and $|I|=i_1+\cdots +i_m$.

  Theorem \ref{cdmain} as a  surprising result  is well-known for its purely  geometric proof(see Sec 3, \cite{CD}) via a holomorphic vector bundle $\et$ over $\Omega$ associated to $\mathbf{T}$, where the   fiber of $\et$  at $\zz$ is nothing but the first order localization
 $   H_\zz^1=\cap_{i=1}^m \ker(T_i-{z_i})$(see Proposition \ref{hooo} below).  Geometrically reducing  unitary equivalence  of    $\mathbf{T}|_{H_\zz^{n+1}}$  to   criterions  involving  the  curvature of $\et$ and its covariant derivatives   living on the fiber $H_\zz^1$ is not only a key step in the proof of Theorem \ref{cdmain}(see Sec 2, \cite{CD}), but also motivates numerous interesting later works, including the study of   flag structure and curvature inequalities of $\et$ that are of  independent geometric interest(\cite{CD2,flgjfa,MM1,MM2,MP,MR,WZ}), as well as  geometric classification theories for Hilbert modules(\cite{Chen,D,DM2,DMV,DP}).

  In this paper we   show that unitary equivalence of $H_\zz^k$  can be determined by    a    set of Specht-type operators (involving   words in  $T_1,\cdots,T_m$  and $T_1^*,\cdots,T_m^*$) acting on    $H_\zz^1$, which is an operator-theoretic counterpart of the above mentioned geometric reduction and  extends   results from recent works of  Misra, Pal and Reza(\cite{MP,MR}) as well.  In particular, as the  first order localization is much smaller than high order ones, letting $\zz$ run through $\Omega$  our result (applied to   localizations of order $n+1$)   implies  a     refinement of Theorem \ref{cdmain} in light of Remark \ref{test}.

\section{Prelimineries}

For a separable Hilbert space $\mathcal {H}$ and a positive integer
$n$, let $\mathcal {G}r(n,\mathcal {H})$ denote the Grassmann
manifold of all $n$-dimensional subspaces of $\mathcal {H}$.
 A map $E:
 \Omega\rightarrow \mathcal {G}r(n,\mathcal {H})$ is called a
 holomorphic curve if  there
 exists  $n$ holomorphic $\mathcal {H}$-valued functions $\gamma_1, \cdots,
 \gamma_n$, called a holomorphic frame,   such that $E(\zz)=span\{\gamma_1(\zz), \cdots,
 \gamma_n(\zz)\}$ for every $\zz\in\Omega$.  In particular, this defines a holomorphic Hermitian vector bundle of rank $n$ over $\Omega$ whose Hermitian metric   is  just the inner product of $\mathcal {H}$.  The following important observation of Cowen and Douglas(Proposition 1.11, \cite{CD})   asserts that the family of joint eigen-spaces of an operator tuple $\mathbf{T}\in\bnn$ is such an bundle.
\begin{prop}\label{hooo}
For any $\mathbf{T}\in\bnn$,  the map $ \zz\mapsto \cap_{i=1}^m \ker(T_i-{z_i})$ defines a holomorphic curve.\end{prop}  We denote this holomorphic curve by  $E(\mathbf{T})$. In particular,   there exists holomorphic $\mathcal {H}$-valued functions
  $\gamma_1,\cdots,\gamma_n$ over $\Omega$ such that
 $\cap_{i=1}^m \ker(T_i-{z_i})=span\{\gamma_1(\zz),\cdots,\gamma_n(\zz)\}$. Now  we   further
represent the high order localization  $H_\zz^k= \cap_{|I|=k}\ker (\mathbf{T}-\zz)^I$  in terms of    derivatives of these functions.

For a fixed multi-index $I=(i_1,\cdots,i_m)$, set $I!=i_1!\cdots i_m!$ and  $\partial^I=\p^{i_1}_{ 1}\p^{i_2}_{2}\cdots\p^{i_m}_{m}$ where $\p_i$ denotes differentiation with respect to $z_i$, $1\leq i\leq m$. Given another multi-index $J=(j_1,\cdots,j_m)$, we say $I\geq J$ if $i_k\geq j_k$ for all $1\leq k\leq m$,  and the index $(i_1-j_1,\cdots,i_m-j_m)$ is denoted by $I-J$.

As $\cap_{i=1}^m \ker(T_i-{z_i})=span\{\gamma_1(\zz),\cdots,\gamma_n(\zz)\}$,   the identity $$(T_j-z_j)\gamma_i(\zz)=0$$   holds for   every $1\leq i,  j \leq m$, which, combined with standard differentiation computations  via the Leibnitz rule(or see p.470, \cite{CS}), yields

 \begin{equation}\label{diffe} (\mathbf{T}-\zz)^I\p^J\gamma_i(\mathbf{z}) =\left\{\begin{array}{l}{\frac{J!}{(J-I)!}\gamma_i(\mathbf{z}), J\geq I} \\ {0, \quad\quad \mathrm{otherwise}}\end{array}\right.\end{equation}

\begin{lem}\label{up}Given an operator tuple $\mathbf{T}=(T_1,\cdots, T_m)\in\bnn$ and a fixed holomorphic frame $\{\gamma_1,\cdots,\gamma_n\}$ for $\et$, it holds that for any positive integer $k$, $$\cap_{|I|=k}\ker (\mathbf{T}-\zz)^I=span_{|I|\leq k-1}\{\p^I\gamma_i(\mathbf{z}), 1\leq i\leq n\}$$\end{lem}
\begin{proof}That $span_{|I|\leq k-1}\{\p^I\gamma_i(\mathbf{z}), 1\leq i\leq n\}\subseteq \cap_{|I|=k}\ker (\mathbf{T}-\zz)^I$ trivially follows from (\ref{diffe}) and we prove $\cap_{|I|=k}\ker (\mathbf{T}-\zz)^I\subseteq span_{|
I|\leq k-1}\{\p^I\gamma_i(\mathbf{z}), 1\leq i\leq n\}$ by induction.

The conclusion trivially holds when $k=1$ and we suppose it holds for some  $k$.
Now fix
$x\in \cap_{|I|=k+1}\ker (\mathbf{T}-\zz)^I$, then for any $I$ such that $|I|=k$, $(\mathbf{T}-\zz)^Ix$ lies in $\cap_{|I|=1}\ker (\mathbf{T}-\zz)^I$, which is   the joint eigen-space spanned by
 $\{\gamma_1(\mathbf{z}), \cdots,\gamma_n(\mathbf{z})\}$. Hence we  get a collection of complex numbers $\{a^J_i| 1\leq i\leq n, |J|=k\}$  such that
 \begin{equation}\label{09}(\mathbf{T}-\zz)^Ix=\sum_{i=1}^n a^I_i\gamma_i(\mathbf{z})\end{equation} whenever $|I|=k$.

 We claim that the vector $$x-\sum_{|J|=k}\sum_{i=1}^n\frac{a_i^J}{J!}\p^J\gamma_i(\mathbf{z})$$ lies in $\cap_{|I|=k}\ker (\mathbf{T}-\zz)^I$. Then the induction hypothesis   together with the claim implies that
  $x\in span_{|J|\leq k}\{\p^J\gamma_i(\mathbf{z}), 1\leq i\leq n\}$, which gives the conclusion for $k+1$ and   completes the induction.

 To verify the claim, we fix a multi-index $I$ such that $|I|=k$, then   for any   multi-index $J$ with $|J|=k$,  it holds by (\ref{diffe}) that
 $$  (\mathbf{T}-\zz)^I\p^J\gamma_i(\mathbf{z}) =\left\{\begin{array}{l}{I!\gamma_i(\mathbf{z}), I=J} \\ {0, \quad\quad I\neq J}\end{array}\right.$$ This implies
\begin{equation}\label{10}(\mathbf{T}-\zz)^I\sum_{|J|=k}\sum_{i=1}^n\frac{a_i^J}{J!}\p^J\gamma_i(\mathbf{z})=
(\mathbf{T}-\zz)^I \sum_{i=1}^n\frac{a_i^I}{I!}\p^I\gamma_i(\mathbf{z})=\sum_{i=1}^n a^I_i\gamma_i(\mathbf{z}), \end{equation}  and the claim   follows by comparing  (\ref{09})and (\ref{10}).
\end{proof}

%\begin{cor}\label{ajj}For fixed positive integer $d<m$ and a subset $A=\{a_1,\cdots,a_d\}\subseteq\{1,2,\cdots,m\}$, let $\Lambda=\{I=(i_1,\cdots,i_m)|i_{a_1}=i_{a_2}=\cdots=i_{a_d}=0\}$, then it holds that
%
%\begin{equation}\label{upp} \cap_{|I|=k,I\in \Lambda}\ker(\mathbf{T}-\zz)^I  \cap_{l\in A}\ker(T_l-z_l)=span_{|J|\leq k-1,J\in\Lambda}\{\p^J\gamma_i(\mathbf{z}), 1\leq i\leq n\}.\end{equation}
%
%%In particular, when $k=2$ and  $A$ has cardinality $m-1$, it holds that
%%\begin{equation}\label{uppp}  \cap_{t\neq r}\ker(N_t-z_t)\cap\ker (N_r-z_r)^2=span
%%\{\gamma_i(\zz), \p_r\gamma_i(\mathbf{z}), 1\leq i\leq n\}\end{equation} for every $1\leq r\leq m$.
%
%\end{cor}
%
%\begin{proof} Note that $T_l-z_l$   annihilates  $\p^J\gamma_i(\mathbf{z})$ whenever $l\in A$ and $J\in \Lambda$, the corollary follows from  a  straightforward modification of above  proof  for   Lemma \ref{up}.\end{proof}

 \bigskip

We end this section with two    elementary lemmas.

 \begin{lem} \label{norro}  Let $E$ be a holomorphic curve   over $\Omega$ of rank $n$. For any point $\zz_0$ in $\Omega$, there exist a holomorphic frame
 $\{\gamma_i(\zz)\}_{i=1}^n$  for $E$ over an  neighborhood   of $\zz_0$ on which  $\langle \gamma_i(\zz), \gamma_j(\zz_0)\rangle=\delta_{ij}$ for all $1\leq i,j\leq n$, where $\delta_{ij}$ is the Kronecker symbol. \end{lem}

 %By shrinking $\Omega$, we can always assume that the normalized frame exists on entire $\Omega$.

\begin{lem}\label{zw}Let $f(\zz,\ww)$ be a function on $\Omega\times\Omega$ which is holomorphic in $\zz$ and anti-holomorphic in $\ww$. If $f(\zz,\zz)=0$ for all $\zz\in\Omega$, then $f(\zz,\ww)$ vanishes identically on $ \Omega\times\Omega$.\end{lem}
Lemma \ref{norro}(see Lemma 2.4, \cite{CD}) asserts that a holomorphic curve   always admits   a   holomorphic frame   normalized at a single point which is called a \emph{normalized frame}, and Lemma \ref{zw} is standard   which will be useful in dealing with normalized frames     later.

\section{Main result }

 Throughout this section, $k$ will be a fixed positive integer.  We show that for an operator tuple  $\mathbf{T}=(T_1,\cdots, T_m) \in\bnn$,  unitary equivalence of its $k$-th order  localization $\mathbf{T}_{H_\mathbf{z}^k}$  can be tested  by a collection of ``Specht-type" operators acting  on the first order localization ${H_\mathbf{z}^1}=\cap_{i=1}^m \ker(T_i-{z_i}).$

Precisely, for each  multi-index $I=(i_1,\cdots, i_m)$, and $\mathbf{z}=(z_1,\cdots, z_m)\in\Omega$,   set $$N_\mathbf{z}^I:=(T_1-z_1)^{i_1}(T_2-z_2)^{i_2}\cdots (T_m-z_m)^{i_m}|_{H_\mathbf{z}^k}, $$ and for two fixed  multi-indices $I,J$, set   $$K_\zz^{IJ}:=P_{H_\mathbf{z}^1}[N_\mathbf{z}^I({N_\mathbf{z}^J)^*}]|_{H_\mathbf{z}^1}.$$
Here as $N_\mathbf{z}^I({N_\mathbf{z}^J)^*}$  does not necessarily leave $H_\zz^1$ invariant,   $P_{H_\mathbf{z}^1}$ is imposed to make $K_\zz^{IJ}$ live in $H_\zz^1$.

 The main result of this paper as follows gives the Specht-type classification for $\mathbf{T}_{H_\mathbf{z}^k}$ in terms of $K_\zz^{IJ}$:

%Set $N=(T-w)|_{\ker(T-w)^{n+1}}$ and  $N_w^{ij}=P_{\ker T-w}N^i{N^*}^j|_{\ker T-w}$.

\begin{thm}\label{mmain}Given operator tuples $\mathbf{T} $ and $\mathbf{\w T}$ in $\mathcal {B}_n^m(\Omega) $,
their localizations $\mathbf{T}|_{H_\mathbf{z}^k}$ and $\mathbf{\w T}|_{\widetilde{H}_\mathbf{z}^k}$ are unitarily equivalent if and only if operator tuples  $\{K_\zz^{IJ},1\leq |I|,|J|\leq k-1\}$ and $\{\w K_\zz^{IJ},1\leq |I|,|J|\leq k-1\}$  are unitarily equivalent.

\end{thm}

% Note that unitary equivalence of  $\mathbf{T}|_{H_\mathbf{z}^k}$ and $\mathbf{\w T}|_{\widetilde{H}_\mathbf{z}^k}$  is the same as that of $(T_1-z_1,\cdots, T_m-z_m)|_{H_\mathbf{z}^k}$ and $(\w T_1-z_1,\cdots, \w T_m-z_m)|_{\widetilde{H}_\mathbf{z}^k}$),

The first nontrivial case  $k=2$(where one can check that $H_\zz^1$ is   actually invariant under $N_\mathbf{z}^I({N_\mathbf{z}^J)^*}$)  were studied  by   Misra, Pal and Reza  in their works on curvature inequalities for  $\mathcal {B}_n^m(\Omega) $ \cite{MP,MR}. If we assume further that $n=1$, then  $K_\zz^{IJ}$ as  a scaler    exactly  equals the trace of  $N_\mathbf{z}^I({N_\mathbf{z}^J)^*}$  so  it is eligible to call   $K_\zz^{IJ}$ a Specht-type invariant.

 \bigskip

With Theorem \ref{mmain}, we immediately get the following refinement of Theorem \ref{cdmain}
    \begin{thm}\label{refinement}Operator tuples $\mathbf{T} $ and $\mathbf{\w T}$ in $\mathcal {B}_n^m(\Omega) $  are unitarily equivalent if and only if  operator tuples $\{K_\zz^{IJ},1\leq |I|,|J|\leq n \}$ and $\{\w K_\zz^{IJ},1\leq |I|,|J|\leq n \}$  are unitarily equivalent for every $\zz$ in $\Omega$.\end{thm}

%Note that if $I=J=(0,\cdots,0)$, $K_\zz^{IJ}$ is nothing but  the identity operator, while when $|I|\neq 0$, both $P_{H_\mathbf{z}^1}N_\mathbf{z}^I |_{H_\mathbf{z}^1}$ and $P_{H_\mathbf{z}^1} ({N_\mathbf{z}^I)^*}|_{H_\mathbf{z}^1}$ are   zero operators since $N^I_\zz$ annihilates   $\bg(\zz)$ which spans $H_\zz^1$.

The proof of Theorem \ref{mmain} will be given in the end of this section after some preparations. Before proceeding, we fix some notations and  conventions in elementary linear algebra.

(i) \emph{``Inner product" of matrices}: Let $A=[a_{ij}]_{m\times n}$ and $B=[b_{ij}]_{n \times p}$ be two  matrices with entries  $a_{ij}, b_{ij}$ lying in a Hilbert space(whose inner product is denoted by $\langle,\rangle$). Let $\langle A, B\rangle$ denotes the numerical matrix  $E=[e_{ij}]_{m \times p}$ given by $e_{ij}=\sum_{k=1}^n \langle a_{ik},b_{kj}\rangle.$ If $C, D$ are numerical matrices,   then $\langle CA, BD\rangle$=$C\langle A, B\rangle \overline{D}$.

With this notation, if $\bg=(\gamma_1,\cdots,\gamma_n)$ is a holomorphic frame for a rank $n$ holomorphic curve, then   its Gram matrix can be written as $\langle\bg^T(\zz),\bg(\zz)\rangle$. Moreover,  $\bg$ is normalized at a point $\zz_0$ if and only if $\langle\bg^T(\zz),\bg(\zz_0)\rangle=I$ identically.
%. In particular,  if  $\ba=(a_1,\cdots,a_n)$ and $\bb=(b_1,\cdots,b_n)$ are two rows of complex numbers, $\langle \sum a_i\gamma_i, \sum b_i\gamma_i\rangle=\langle \ba\bg^T, \bg\bb^T\rangle=\ba\langle \bg^T, \bg\rangle{\bb}^*$

(ii)\emph{Representation of linear maps}: We adopt the ``left action" convention regarding to representing matrices for linear maps. Precisely, let $\Phi$ be a linear map on a linear space spanned by  $\bg=(\gamma_1,\cdots,\gamma_n)$, then  a matrix  $A=[a_{ij}]$ represents $\Phi$ if  $\Phi\gamma_i=\sum a_{ij}\gamma_j$, or in other words, $\Phi\bg^T=(\Phi\gamma_1,\cdots,\Phi\gamma_n)^T=A(\gamma_1,\cdots,\gamma_n)^T$. If another linear map $\Psi$ is represented by $B=[b_{ij}]$, then $\Phi\Psi$ is represented by $BA$(not $AB$, which corresponds to ``right action" convention).

%\begin{proof}Let $B$ be the representing matrix for $\Phi^*$ with respect to $\bg$. For any two rows of complex numbers $\ba=(a_1,\cdots,a_n)$ and $\bb=(b_1,\cdots,b_n)$, it holds that $$\langle\Phi(\ba\bg^T),\bb\bg^T\rangle=\langle\ba\Phi\bg^T,\bb\bg^T\rangle=\langle\ba A\bg^T,\bb\bg^T\rangle
%=\ba A\langle  \bg^T,\bg^T\rangle\bb^*=\ba AH\bb^*.$$

%On the other hand,
%$$\langle\ba\bg^T,\Phi^*(\bb\bg^T)\rangle=\langle\ba\bg^T,\bb\Phi^*\bg^T\rangle=\langle\ba \bg^T,\bb B\bg^T\rangle
%=\ba\langle  \bg^T,\bg^T\rangle B^*\bb^*=\ba  HB^*\bb^*.$$ Hence $\ba AH\bb^*=\ba  HB^*\bb^*$ which implies
%$$AH=HB^*$$ as $\ba$ and $\bb$ are arbitrarily chosen. Note that $H$ is self-adjoint, it follows that $B=(H^{-1}AH)^*=HA^*H^{-1}$.
%
%
%
%
%
%
%\end{proof}

Throughout this section we will work with normalized frames of $\et$.  In the single variable case $m=1$, the normalized frame was used in the study of geometric theory of  $\mathcal {B}_n^1(\Omega)$ to identify  localization order of  an   operator  with ``contact order" of  a holomorphic curve(see Section 2, \cite{CD} for details), and  here we need the following   variation     for $\bnn$ as a preparation before proving Theorem \ref{mmain}.

\begin{thm}\label{po}The followings are equivalent

 (i)  $H_\mathbf{z}^k$ and $\widetilde{H}_\mathbf{z}^k$ are unitarily equivalent;

 (ii)there exists holomorphic frames $ \mbox{\boldmath$\gamma$}  $ and $ \mbox{\boldmath$\w\gamma$} $(whose Gram matrices are denoted by $H$ and  $\widetilde{H}$)for  $E(\textbf{T})$  and ${E}({\widetilde{\textbf{T}}})$   such that $\p^I\overline{\p}^J H=\p^I\overline{\p}^J \widetilde{H}$  at  $\mathbf{z}$ for all  $|I|,|J|\leq k-1$;

 (iii)there exists holomorphic frames $ \mbox{\boldmath$\gamma$}  $ and $ \mbox{\boldmath$\w\gamma$} $ normalized at $\mathbf{z}$ such that $\p^I\overline{\p}^J H=\p^I\overline{\p}^J \widetilde{H}$  at  $\mathbf{z}$ for all   $|I|,|J|\leq k-1$;

(iv) there exists holomorphic frames $ \mbox{\boldmath$\gamma$}  $ and $ \mbox{\boldmath$\w\gamma$} $ normalized at $\mathbf{z}$ and a constant unitary matrix $U$,  such that $\p^I\overline{\p}^J H=U(\p^I\overline{\p}^J \widetilde{H})U^*$ at  $\mathbf{z}$ for all  $|I|,|J|\leq k-1$;

(v) For any holomorphic frames $ \mbox{\boldmath$\gamma$}  $ and $ \mbox{\boldmath$\w\gamma$} $ normalized at $\mathbf{z}$, there exists a constant unitary matrix $U$,  such that $\p^I\overline{\p}^J H=U(\p^I\overline{\p}^J \widetilde{H})U^*$ at  $\mathbf{z}$ for all   $|I|,|J|\leq k-1$.
 \end{thm}

We need two elementary lemmas before the proof of Theorem \ref{po} and Theorem \ref{mmain}.
\begin{lem}\label{kk}Let $\mbox{\boldmath$\gamma$}=\{\gamma_1,\cdots,\gamma_n\}$ and $\mbox{\boldmath$ \beta$}=\{\beta_1,\cdots,\beta_n\}$ be two holomorphic frames of  a holomorphic curve over $\Omega$  such $\mbox{\boldmath$\gamma$}$ is normalized at a point $\zz_0$. Then  $\mbox{\boldmath$\beta$}$ is normalized at $\zz_0$ if and only if its transition function with $\mbox{\boldmath$\gamma$}$ is a constant unitary matrix.
\end{lem}

\begin{proof}
For one direction, let $U$ be a constant unitary matrix and $\mbox{\boldmath$\beta$}^T=U\mbox{\boldmath$\gamma$}^T, $ then for $\zz$ in $\Omega$, $$\langle \mbox{\boldmath$\beta$}^T(\zz),\mbox{\boldmath$\beta$(\zz)}\rangle=U\langle \mbox{\boldmath$\gamma$}^T(\zz),\mbox{\boldmath$\gamma$}(\zz)\rangle U^*, $$ where $U^*$ denotes the conjugate transpose of $U$. The above identity  can be refined by Lemma \ref{zw} into
$$\langle \mbox{\boldmath$\beta$}^T(\zz),\mbox{\boldmath$\beta$(\ww )}\rangle=U\langle \mbox{\boldmath$\gamma$}^T(\zz),\mbox{\boldmath$\gamma$}(\ww)\rangle U^*$$ for all  $\zz,\ww\in \Omega$.
As  $\mbox{\boldmath$\gamma$}$ is normalized at $\zz_0$, $\langle \mbox{\boldmath$\gamma$}^T(\zz),\mbox{\boldmath$\gamma$}(\zz_0)\rangle=I$,  hence  by setting $\ww=\zz_0$  the above equation   becomes
%$$\langle \mbox{\boldmath$\w\gamma$}^T(\zz),\mbox{\boldmath$\w\gamma$(\zzz)}\rangle=U\langle \mbox{\boldmath$\w\gamma$}^T(\zz),\mbox{\boldmath$\w\gamma$}(\ww)\rangle U^*.$$
%$\langle \mbox{\boldmath$\w\gamma$}^T(\zz),\mbox{\boldmath$\w\gamma$(\zz_0)$}\rangle=I$

$$\langle \mbox{\boldmath$\beta$}^T(\zz),\mbox{\boldmath$\beta$ }(\zzz )\rangle=U\langle \mbox{\boldmath$\gamma$}^T(\zz),\mbox{\boldmath$\gamma$}(\zzz)\rangle U^*=UU^*=I.$$ Hence $\mbox{\boldmath$\beta$}$ is also normalized at $\zzz$.

Conversely, let $U(\zz)$ be the transition   function  of $\mbox{\boldmath$\beta$}$ with  $\mbox{\boldmath$\gamma$}$, then $U(\zz)$ is holomorphic  and   $$\langle \mbox{\boldmath$\beta$}^T(\zz),\mbox{\boldmath$\beta$(\zz)}\rangle=U(\zz)\langle \mbox{\boldmath$\gamma$}^T(\zz),\mbox{\boldmath$\gamma$}(\zz)\rangle U(\zz)^*,$$
which can be refined into
 $$\langle \mbox{\boldmath$\beta$}^T(\zz),\mbox{\boldmath$\beta$(\ww)}\rangle=U(\zz)\langle \mbox{\boldmath$\gamma$}^T(\zz),\mbox{\boldmath$\gamma$}(\ww)\rangle U(\ww)^*.$$
 If $\mbox{\boldmath$\beta$}$ is also  normalized at $\zzz$, then $\langle \mbox{\boldmath$\beta$}^T(\zz),\mbox{\boldmath$\beta$ }(\zzz )\rangle=\langle \mbox{\boldmath$ \gamma$}^T(\zz),\mbox{\boldmath$ \gamma$ }(\zzz )\rangle=I$ and the above equation becomes(by setting $\ww=\zz_0$)
  $$I=\langle \mbox{\boldmath$\beta$}^T(\zz),\mbox{\boldmath$\beta$}(\zzz)\rangle=U(\zz)\langle \mbox{\boldmath$\gamma$}^T(\zz),\mbox{\boldmath$\gamma$}(\zzz)\rangle U(\zzz)^*=U(\zz)U(\zzz)^*.$$
  This gives $U(\zz)=U^{-1}(\zzz)^*$, so $U(\zz)$ is a constant unitary matrix.
\end{proof}

The  following  lemma on elementary linear algebra is standard and we omit the proof.
\begin{lem}\label{adj}Let $\Phi$ be a linear operator on a finite dimensional Hilbert space and   $\bg=\{\gamma_1,\cdots,\gamma_n\}$ be a base whose Gram matrix is $H$. If $\Phi$ is   represented by a matrix $A$ with respect to $\bg$,  then its adjoint operator    is represented by $HA^*H^{-1}$.
\end{lem}

Proof of  Theorem \ref{po}:
\begin{proof}
(ii)$\Rightarrow$(i) Write $ \mbox{\boldmath$\gamma$}=\{ {\gamma}_1,\cdots, \gamma_n\}  $ and $ \mbox{\boldmath$\w\gamma$}=\{ {\w\gamma}_1,\cdots, \w\gamma_n\}  $, hence
by Lemma (\ref{up}),
 $H_z^k=span_{|J|\leq k-1}\{\p^J\gamma_i(\mathbf{z}), 1\leq i\leq n\}$ and  $\widetilde{H}_z^k=span_{|J|\leq k-1}\{\p^J\widetilde{\gamma}_i(\mathbf{z}), 1\leq i\leq n\}$. Let $\Phi$ be the linear map from  $H_\mathbf{z}^k$ to  $\widetilde{H}_\mathbf{z}^k$ defined by $$\Phi  \p^J\gamma_i(\mathbf{z}):= \p^J\widetilde{\gamma}_i(\mathbf{z}),|J|\leq k-1,$$ then $\Phi$ implements a unitary equivalence between  $H_\mathbf{z}^k$ and $\widetilde{H}_\mathbf{z}^k$.

 In fact,  $\Phi$ trivially intertwines $T_l-z_l$ and $\w T_l-z_l$(hence intertwines $T_l$ and $\w T_l$), $1\leq l\leq m$    as their actions on $\p^{I }\mbox{\boldmath$\gamma$}(\zz)$ and  $\p^{I }\mbox{\boldmath$\w\gamma$}(\zz)$    follows the same rule (\ref{diffe}).
 Moreover,       $\p^I\overline{\p}^JH=[\langle\p^I\gamma_i, {\p^J}\gamma_j\rangle]_{1\leq i,j\leq n}$ and $\p^I\overline{\p}^J\w H=[\langle\p^I\w\gamma_i, {\p^J}\w\gamma_j\rangle]_{1\leq i,j\leq n} $ since the frames are holomorphic, hence the condition (ii)      implies
 $$\langle\p^I\gamma_i,\p^J\gamma_j\rangle=\langle\p^I\w\gamma_i,\p^J\w\gamma_j\rangle$$ for every $1\leq i,j\leq n$ and $|I|,|J|\leq k-1$ at $\zz$, so $\Phi$ is isometric as well.

(iii)$\Rightarrow$(ii) Trivial.

(i)$\Rightarrow$(iii) Let $\Phi$ be  a unitary  operator from $H_\mathbf{z}^k$ to  $\widetilde{H}_\mathbf{z}^k$ which implements the unitary equivalence. We   show that  there exists  holomorphic   frames $ \mbox{\boldmath$\gamma$}$ and $ \mbox{\boldmath$\w\gamma$}$  normalized at $\mathbf{z}$ such that

 \begin{equation}\label{pll}\Phi \p^I\gamma_i(\mathbf{z})=\p^I\w\gamma_i(\mathbf{z})\end{equation} for all $|I|\leq k-1, 1\leq i\leq n$, and (iii) will  follow since $\Phi$ is isometric.

We begin with arbitrary  fixed holomorphic frames
 %We claim that there exists holomorphic frames $ \mbox{\boldmath$\gamma$}=\{ {\gamma}_1,\cdots, \gamma_n\}  $ and $ \mbox{\boldmath$\w\gamma$}=\{ {\w\gamma}_1,\cdots, \w\gamma_n\}  $  for $E$ and $\w E$  normalized at $\mathbf{z}$ such that $\Phi\gamma_i(\mathbf{z})=\w\gamma_i(\mathbf{z})$ for every $1\leq i\leq n$.
  $ \mbox{\boldmath$\gamma$}=\{ {\gamma}_1,\cdots, \gamma_n\}  $ and $ \mbox{\boldmath$\w\gamma$}=\{ {\w\gamma}_1,\cdots, \w\gamma_n\}  $   for $E(\textbf{T})$  and ${E}({\widetilde{\textbf{T}}})$  normalized at $\mathbf{z}$.
 As $\Phi$   intertwines $T_l-z_l$ and $\w T_l-z_l$, it  maps the joint eigen-space of $\textbf{T}$ spanned by $ \mbox{\boldmath$\gamma$}(\zz)   $ to corresponding one of ${\widetilde{\textbf{T}}}$ spanned by $ \mbox{\boldmath$\w\gamma$}(\zz)$,  hence there exists an $n\times n$ matrix $U$ such that
$$ \Phi\mbox{\boldmath$ \gamma$}^T(\zz)=U\mbox{\boldmath$\w\gamma$}^T(\zz),$$ and $U$  is unitary since both $ \mbox{\boldmath$\gamma$} $ and $ \mbox{\boldmath$\w\gamma$}$ are normalized at $\mathbf{z}$. By Lemma \ref{kk}, $U\w\bg^T$ is again a normalized frame at $\mathbf{z}$, so we can replace $ \w\bg^T  $ by $U\w\bg^T$ which  gives (\ref{pll}) in case $|I|=0$.

 Now we check (\ref{pll}) by induction on $|I|$. Suppose (\ref{pll})   holds with all $|I|\leq l$ for some $l$. For any $I$ with $|I|= l+1$, it holds that  \begin{equation}\label{mnb}(\w T_q-z_q)(\Phi \p^I\gamma_i(\mathbf{z})-\p^I\w\gamma_i(\mathbf{z}))=0\end{equation} for every $1\leq q\leq m$.

In fact, by the intertwining property of $\Phi$,
$$(\w T_q-z_q)(\Phi \p^I\gamma_i(\mathbf{z})-\p^I\w\gamma_i(\mathbf{z}))=
\Phi(T_q-z_q)\p^I\gamma_i(\mathbf{z})-(\w T_q-z_q)\p^I\w\gamma_i(\mathbf{z}).$$ Write $I=(i_1,\cdots,i_q,\cdots,i_m\}$, then in case  $i_q=0$, both  $(T_q-z_q)\p^I\gamma_i(\mathbf{z})$ and $(\w T_q-z_q)\p^I\w\gamma_i(\mathbf{z})$ vanishes hence (\ref{mnb}) trivially holds. In case $i_q\geq 1$,  $(T_q-z_q)\p^I\gamma_i(\mathbf{z})=i_q\p^{I'}\gamma_i(\mathbf{z})$ and $(\w T_q-z_q)\p^I\gamma_i(\mathbf{z})=i_q \p^{I'}\w\gamma_i(\mathbf{z})$ where   $I'=(i_1,\cdots,i_{q}-1,\cdots,i_m)$, hence (\ref{mnb}) follows from the induction hypothesis.

Moreover, we observe that  \begin{equation}\label{pzx}\langle\Phi \p^I\gamma_i(\mathbf{z})-\p^I\w\gamma_i(\zz), \w\gamma_j(\mathbf{z})\rangle=0\end{equation} for every $1\leq j\leq n$.

In fact, as $\Phi$ is isometric and the frames are normalized at $\zz$, it holds that
\footnotesize
$$\langle\Phi \p^I\gamma_i(\mathbf{z})-\p^I\w\gamma_i(\zz), \w\gamma_j(\mathbf{z})\rangle=
\langle\Phi \p^I\gamma_i(\mathbf{z}), \Phi\gamma_j(\mathbf{z})\rangle
-\langle \p^I\w\gamma_i(\mathbf{z}), \w\gamma_j(\mathbf{z})\rangle=\langle  \p^I\gamma_i(\mathbf{z}), \gamma_j(\mathbf{z})\rangle
-\langle \p^I\w\gamma_i(\mathbf{z}), \w\gamma_j(\mathbf{z})\rangle=0-0=0$$ \normalsize as desired.

  Now (\ref{mnb}) implies that $\Phi \p^I\gamma_i(\mathbf{z})-\p^I\w\gamma_i(\zz) $    lies in  $\cap_{i=1}^m \ker(\w T_i-{z_i})$ which is spanned by $\{ {\w\gamma}_1(\zz),\cdots, \w\gamma_n(\zz)\}  $, hence (\ref{pzx}) forces $\Phi \p^I\gamma_i(\mathbf{z})-\p^I\w\gamma_i(\zz)=0$, concluding the induction.

(v)$\Rightarrow$(iv) Trivial.

(iv)$\Rightarrow$(iii) Let $ \mbox{\boldmath$\gamma$}  $ , $ \mbox{\boldmath$\w\gamma$} $   and  $U$ be as given by (iv), then   $\mbox{\boldmath$\w\sigma$}^T:=U\mbox{\boldmath$\w\gamma$}^T $ is again a holomorphic frame for  ${E}({\widetilde{\textbf{T}}})$ normalized at $\zz$ by Lemma \ref{kk}. Moreover,
$$\p^I\overline{\p}^J\langle \mbox{\boldmath$\sigma$}^T,\mbox{\boldmath$\sigma$}\rangle=\p^I\overline{\p}^J (U\langle \mbox{\boldmath$\w\gamma$}^T,\mbox{\boldmath$\w\gamma$}\rangle U^*)=U(\p^I\overline{\p}^J  \langle \mbox{\boldmath$\w\gamma$}^T,\mbox{\boldmath$\w\gamma$}\rangle) U^*$$ holds in a neighborhood of $\zz$ which, specifying at $\zz$,  equals $\p^I\overline{\p}^J\langle \mbox{\boldmath$\gamma$}^T,\mbox{\boldmath$\gamma$}\rangle$ by (iv),  so $ \mbox{\boldmath$\gamma$}  $ and   $\mbox{\boldmath$\sigma$}$ meets (iii).

(iii)$\Rightarrow$(v)Fix holomorphic frames $ \mbox{\boldmath$\beta$}  $ and $ \mbox{\boldmath$\w\beta$} $ normalized at $\mathbf{z}$ with properties given by (iii), then for arbitrarily chosen holomorphic frames $ \mbox{\boldmath$\gamma$}  $ and $ \mbox{\boldmath$\w\gamma$} $ normalized at $\mathbf{z}$, their exists, by Lemma \ref{kk}, constant unitary matrices $V$ and $\w V$ such that  $\mbox{\boldmath$\gamma$}^T=V\mbox{\boldmath$\beta$}^T $ and
 $\mbox{\boldmath$\w\gamma$}^T=\w V\mbox{\boldmath$\w\beta$}^T,$  which gives

$$\p^I\overline{\p}^J\langle \mbox{\boldmath$\gamma$}^T,\mbox{\boldmath$\gamma$}\rangle=V(\p^I\overline{\p}^J  \langle \mbox{\boldmath$\beta$}^T,\mbox{\boldmath$\beta$}\rangle) V^*=V(\p^I\overline{\p}^J  \langle \mbox{\boldmath$\w\beta$}^T,\mbox{\boldmath$\w\beta$}\rangle) V^*=V\w V^*(\p^I\overline{\p}^J  \langle \mbox{\boldmath$\w\gamma$}^T,\mbox{\boldmath$\w\gamma$}\rangle) \w VV^*$$ at $\zz$. The proof is completed by taking $U=V\w V^*$. \end{proof}

Finally we give the proof Theorem \ref{mmain}.

\begin{proof}
\bigskip
%Note that if $I=J=(0,\cdots,0)$, $K_\zz^{IJ}$ is nothing but  the identity operator, while when $|I|\neq 0$, both $P_{H_\mathbf{z}^1}N_\mathbf{z}^I |_{H_\mathbf{z}^1}$ and $P_{H_\mathbf{z}^1} ({N_\mathbf{z}^I)^*}|_{H_\mathbf{z}^1}$ are   zero operators since $N^I_\zz$ annihilates   $\bg(\zz)$ which spans $H_\zz^1$.
We fix a holomorphic frame $ \mbox{\boldmath$\gamma$}=\{ {\gamma}_1,\cdots, \gamma_n\}  $ for $E(\mathbf{T})$ normalized at $\zz$ and begin by calculating the matrix representation of $K^{IJ}_\zz$  with respect to the base $\bg(\zz)$ of $H_\zz^1$(and $\w K_\zz^{IJ}$  follows in the same way), which will be read out from the representing matrix for  $N_{\zz}^I{N_{\zz}^J}^*$ with respect to the   base  $\{\p^K\gamma_i(\mathbf{z}), 1\leq i\leq n, |K|\leq k-1\}$ of $H_\mathbf{z}^k$.

 Let  $L$ be the cardinality of the multi-index set  $\{K, |K|\leq k-1\}$, then  $\dim H_\mathbf{z}^k=nL$($L$ can be   worked out  via binomial coefficients but we do not need the precise value). The Gram matrix   for $\{\p^K\gamma_i(\mathbf{z}), 1\leq i\leq n,\quad |K|\leq k-1\}$, denoted by $\mathbf{H}$,  is an $L\times L$ block matrix $[H_{IJ}]_{0\leq |I|,|J|\leq k-1}$ in which each block is an $n\times n$ matrix $H_{IJ} :=[\langle\p^I\gamma_i(\zz),\p^J\gamma_j(\zz)\rangle ]_{1\leq i,j\leq n}$.

In principle, to  precisely locate a particular block $H_{IJ}$ in  $\mathbf{H}$ one need to assign   an ordering   for the multi-indices, that is, a bijection $\sigma$ from the set  $\{K, |K|\leq k-1\}$ to $\{0,1,2, \cdots, L-1\}$.
 From now on we fix   a particular ordering(the lexicographic ordering for instance), then we can write  $$[H_{IJ}]_{0\leq |I|,|J|\leq k-1}=[H_{\sigma I,\sigma J}]_{0\leq\sigma I,\sigma J\leq L-1}=[H_{ij}]_{0\leq i,j\leq L-1},$$ where  the terminology ``$I$-th row/column" makes sense(which refers  to ``$\sigma(I)$-th row/column"), and we can freely use the above three representations   in the sequel  which will not cause confusion. In particular,  we assume that $\sigma (0,\cdots,0)=0$, so the block $H_{00}$ is the Gram matrix of $ \mbox{\boldmath$\gamma$}$.

By (\ref{diffe}), for fixed index $J$,  $N_z^J$ maps $\p^J\bg(\zz)$ to $J!\bg(\zz)$,  while for $K\neq J$, linear representation of $N_z^J\p^K\bg(\zz)$ in terms of $\{\p^K\bg(\mathbf{z}), \quad  |K|\leq k-1\}$  has no $\bg(\zz)$-component. Therefore  $N_\zz^J$ has the following block matrix representation with respect to $\{\p^K\bg, |K|\leq k-1\}$($\bg(\zz)$ appears in the  $0$-th   place when we arrange $\{\p^K\bg(\zz), |K|\leq k-1\}$ into a column).

$$N_z^J=
\left(\begin{array}{cccc}{\mathbf{0}}  & {}  &{} & {} \\ {\vdots} & {}  &  { } & {} \\{J!I_n} & {\mathbf{0}} &  {\cdots} & {\mathbf{0}} \\  {\vdots} &{}  & {  } & {} \\{\mathbf{0}} & {}  &{} & {} \end{array}\right)\begin{array}{l}{} \\ {J-{th }} \\ {}\end{array}
$$Here we have not written out all nonzero blocks in $N_z^J$,  since the  only thing we need later is that the $(J,0)$ block $J!I_n$ is the only nonzero block throughout  the 0-th column and $J$-th row.

As the frame $ \mbox{\boldmath$\gamma$}$ is normalized at $\zz$, it holds that $$H_{I0}=[\langle\p^I\gamma_i(\zz), \gamma_j(\zz)\rangle ]_{1\leq i,j\leq n}=\textbf{0},$$(similarly, $H_{0I}=\textbf{0}$) for all $1\leq |I|\leq k-1$ and $H_{00}=I_n$. Therefore, the block matrix   $\mathbf{H}$ is of the form
$$\left(\begin{array}{cccc}{I_n}  & {\textbf{0}}& {\cdots} & {\textbf{0}} \\{\textbf{0}}&  {H_{11}}  &  {\cdots} & {H_{1,L-1}} \\ {\vdots} & {\vdots}& {} & {\vdots} \\{\textbf{0}}& {H_{L-1,1}} & {\cdots} & {H_{L-1,L-1}} \end{array}\right),$$ which in turn implies that its inverse   $\mathbf{G}=[G_{IJ}]_{0\leq |I|,|J|\leq k-1}=[G_{ij}]_{0\leq i,j\leq L-1}$ is of the same form.

%has only one nonzero block  $H_{00}=I_n$ throughout the 0-th row and 0-th column, hence so does its inverse

Now suppose $|I|,|J|
\geq 1$, then by Lemma \ref{adj}, $N_\zz^I{N_\zz^J}^*$ can be represented by
\scriptsize

$$
\left(\begin{array}{cccc}{I_{n}}  & {\textbf{0}}& {\cdots} & {\textbf{0}} \\{\textbf{0}}&  {H_{11}}  &  {\cdots} & {H_{1,L-1}} \\ {\vdots} & {\vdots}& {} & {\vdots} \\{\textbf{0}}& {H_{L-1,1}} & {\cdots} & {H_{L-1,L-1}} \end{array}\right)
\bordermatrix{%
       &        &     &  {  {J}-th}     &\cr
  & \mathbf{0}        & \cdots & J!I_n     &\cdots   & \mathbf{0}\cr     &          &       &\mathbf{0}     & \cr
 &     &    &\vdots  &    &\cr
    &          &        &\mathbf{0}     &  &
}\left(\begin{array}{cccc}{I_{n}}  & {\textbf{0}}& {\cdots} & {\textbf{0}} \\{\textbf{0}}&  {G_{11}}  &  {\cdots} & {G_{1,L-1}} \\ {\vdots} & {\vdots}& {} & {\vdots} \\{\textbf{0}}& {G_{L-1,1}} & {\cdots} & {G_{L-1,L-1}} \end{array}\right)
\left(\begin{array}{cccc}{\mathbf{0}}  & {}  &{} & {} \\ {\vdots} & {}  &  { } & {} \\{I!I_n} & {\mathbf{0}} &  {\cdots} & {\mathbf{0}} \\  {\vdots} &{}  & { } & {} \\{\mathbf{0}} & {}  &{} & {} \end{array}\right)\begin{array}{l}{} \\ \scriptsize{  I-{th }} \\ {}\end{array}$$

$$=\bordermatrix{%
       &        &     &  \scriptsize{  {J}-th}     &\cr
  & \mathbf{0}        & \cdots & J!I_n    &\cdots   & \mathbf{0}\cr     &          &       &      & \cr
 &     &    &\cdots  &    &\cr
    &          &        &      &  &
}\left(\begin{array}{cccc}{\mathbf{0}}  & {}  &{} & {} \\ {I!G_{1I}} & {}  &  { } & {} \\{I!G_{2I}} & { } &  {{\vdots}} & { } \\  {\vdots} &{}  & { } & {} \\{I!G_{L-1,I}} & {}  &{ } & {} \end{array}\right)\begin{array}{l}{} \\  { } \\ {}\end{array}=\left(\begin{array}{cccc}{I!J!G_{JI}}  & { }& {\cdots} & { } \\{ }&  { }  &  {} & { } \\ {\vdots} & { }& {} & { } \\{ }& { } & { } & { } \end{array}\right)$$
\normalsize

Since the frame is normalized at $\zz$, the space $H_\zz^1$  spanned by $\bg(\zz)$ is is orthogonal to the space spanned by  $\{\p^K\bg(\mathbf{z}),  1\leq |K|\leq k-1\}$, which implies that  the   block $ I!J!G_{JI}$ appearing at the left upper corner of $N_\zz^I{N_\zz^J}^*$  exactly  represents $P_{H_\mathbf{z}^1}(N_{\zz}^I{N_{\zz}^J}^*)|_{H_\mathbf{z}^1}$ with respect to the base $\bg(\zz)$ of $H_\mathbf{z}^1$. With  similar notations, $I!J!\w G_{IJ}$ represents $\w K_\zz^{IJ}$ with respect to the normalized frame $\w\bg(\zz)$ of ${E}({\widetilde{\textbf{T}}})$.

\bigskip

Now we are prepared to prove the theorem. For sufficiency, let $\Phi$ be a unitary operator from $H_\mathbf{z}^1$ to $\w H_\mathbf{z}^1$ intertwining $K_\zz^{IJ}$ and $\w K_\zz^{IJ}$ whose  representing matrix    with respect to $\bg(\zz)$ and $\w \bg(\zz)$ is denoted by $U$.  Then $U$ is a unitary matrix as both frames are normalized at $\zz$. Moreover, the intertwining property gives
$$I!J!G_{IJ}U=U(I!J!\w G_{IJ})$$ that is \begin{equation}\label{rpr} G_{IJ} =U \w G_{IJ} U^*\end{equation} for all $1\leq |I|,|J|\leq k-1$ at $\zz$.

Observing that at $\zz$, $G_{00}=\w G_{00}=I_n$ and $G_{I0}=\w G_{I0}=\textbf{0}$ whenever $|I|\neq 0$,  identity \ref{rpr} holds for all $0 \leq |I|,|J|\leq k-1$, which  gives

\begin{equation}\label{ad}   [G_{IJ}]_{0 \leq |I|,|J|\leq k-1}=(U\otimes I_L)
[\widetilde{G}_{IJ}]_{0 \leq |I|,|J|\leq k-1}(U^*\otimes I_L). \end{equation} where $U\otimes I_L$ denotes the diagonal block matrix with $U$ lying on all diagonal blocks.  Taking inverse we get
\begin{equation}\label{aw}    [H_{IJ}]_{0 \leq |I|,|J|\leq k-1}=(U\otimes I_L)
[\widetilde{H}_{IJ}]_{0 \leq |I|,|J|\leq k-1}(U^*\otimes I_L) .\end{equation}

Specifying (\ref{aw}) block-wise we see that at $\zz$,
\begin{equation}\label{rr} H_{IJ} =U \w H_{IJ} U^*\end{equation}
holds for all $0 \leq |I|,|J|\leq k-1$.  Recall that $H_{IJ}=\p^I\overline{\p^J}H_{00}$ and $\w H_{IJ}=\p^I\overline{\p^J}\w H_{00}$, the sufficiency follows from combining (\ref{rr}) and Theorem \ref{po}.

Conversely, if $H_\mathbf{z}^k$ and $\widetilde{H}_\mathbf{z}^k$ are unitarily equivalent, then Theorem \ref{po} implies the existence of a constant unitary matrix $U$ such that (\ref{rr}) holds, which in turn gives, by reversing the above arguments, the intertwining property  (\ref{rpr}). So the unitary operator represented by $U$ with respect to $\bg(\zz)$ and $\w\bg(\zz)$ implements the unitary equivalence of   $\{K_\zz^{IJ}, 1\leq|I|,|J|\leq k-1\}$ and $\{\w K_\zz^{IJ}, 1\leq|I|,|J|\leq k-1\}$.
 \end{proof}

Li Chen

School of Mathematics

Shandong University

  Jinan 250100, China

 Email: lchencz@sdu.edu.cn

\end{document}